\documentclass[a4paper,12pt]{article}

\usepackage[utf8]{inputenc}
\usepackage[english]{babel}
\usepackage{amssymb,amsmath,amsthm}
\usepackage{hyperref}
\usepackage{caption}

\hypersetup{
    colorlinks=true,
    linkcolor=blue,
    filecolor=magenta,      
    urlcolor=cyan,
    }

\newtheorem{theorem}{Theorem}[section]

\newtheorem{proposition}[theorem]{Proposition}
\newtheorem{remark}[theorem]{Remark}
\newtheorem{corollary}[theorem]{Corollary}

\usepackage[left=1in, right=1in, top=1in, bottom=1in]{geometry} 
\usepackage{graphicx} 
\usepackage{tikz}

\numberwithin{equation}{section}

\title{A strong FKG inequality for multiple events}
\author{Nikita Gladkov\thanks{Department of Mathematics,
UCLA, Los Angeles, CA, 90095.~ Email: ~ \texttt{{gladkovna}@ucla.edu}}}
\date{}

\begin{document}

\maketitle

\begin{abstract}
We give an extension of the FKG inequality to the case of multiple events with equal pairwise intersections. We then apply this inequality to resolve Kahn's question on positive associated (PA) measures.
\end{abstract}

\section{Introduction}

The \textit{Fortuin--Kasteleyn--Ginibre} (FKG) \textit{inequality} is an inequality with numerous applications ranging from percolation to graph theory, from poset theory to probability theory \cite[Ch.~6]{AS}, \cite{G}. It generalizes the Harris--Kleitman inequality to a large class of measures on a hypercube. 

Denote by $H_n$ the $n$-dimensional discrete hypercube. We think of it as a distributive lattice, where $\vee$ is a coordinatewise maximum and $\wedge$ is a coordinatewise minimum. Let $\mu$ be a probability measure on $H_n$. Assume $\mu$ satisfies the following \textit{FKG property}:
$$\mu(a\vee b)\mu(a \wedge b) \ge \mu(a)\mu(b) \text{ for all }a, b \in H_n.$$ 
The FKG inequality guarantees nonnegative correlations of events that are closed upwards:

\begin{equation}\label{classicFKG}
P(E_1 \cap E_2) \ge P(E_1)P(E_2) 
\end{equation}

The Harris--Kleitman inequality is the partial case of the FKG inequality for product measures on $H_n$. Measures for which all closed-upwards events correlate nonnegatively are said to have \textit{positive associations} (PA). In other words, the FKG inequality says that all measures with the FKG property are PA. The FKG inequality is used to show that measures arising from random cluster model are PA~\cite{G06}. 
 
We prove a strong version of the Harris--Kleitman inequality for events with equal pairwise intersections (Theorem~\ref{Harris+}). We then use the approach from \cite{kahn2022note} to generalize it to measures with the FKG property (Theorem~\ref{FKG+}). There are generalizations of the Harris--Kleitman and FKG inequalities as well as the more general AD inequality (four functions theorem) \cite{AD} to multiple sets \cite{AK, RS}. Richards \cite{Ri} claimed to prove another generalization of FKG to multiple sets, but the proof has essential gaps \cite{Sa}, so this generalization is proved only with additional restrictive conditions \cite{LS}. As far as we know, our generalization is different from all the others.

We use our inequality to prove a conjecture of Kahn (Theorem~\ref{ans}). Roughly speaking, we establish that the FKG inequality can be extended beyond the random variables which are monotonically determined by underlying independent variables.

\section{Strong Harris--Kleitman inequality}

We say that measure $\mu$ on $H_n$ is a \textit{product measure} if there exist probability measures $\mu_1$, $\mu_2$, \dots, $\mu_n$ on $\{0, 1\}$, such that $\mu$ coincides with the direct product $\mu_1 \times \mu_2 \times \dots \times \mu_n$. Recall the notation $e_2(x_1, \dots, x_k) := \sum_{1 \le i < j \le k} x_ix_j$ for the second symmetric polynomial. We say that a subset of $H_n$ is \textit{closed upwards} if with each vector $v \in H_n$ it also contains all vectors bigger than $v$ in the natural partial order.
Here is our main result:

\begin{theorem}\label{Harris+}
Let $\mu$ be a probability product measure on $H_n$, and 
$$H_n = A \sqcup C_1 \sqcup C_2 \sqcup \dots \sqcup C_k \sqcup B \text{ for }k \ge 2$$
such that all sets of the form $A \cup C_i$ are closed upwards. Then:

\begin{equation}\mu(A)\mu(B) \ge e_2\big(\mu(C_1), \dots, \mu(C_k)\big).\label{E2}\end{equation}

\end{theorem}
\begin{remark}
\rm{For $k=2$ we obtain the Harris--Kleitman inequality \eqref{classicFKG}. Indeed, for probability measure $\mu$ inequality
$$\mu(A)\mu(B) \ge \mu(C_1)\mu(C_2)$$
 is equivalent to the positive correlation of closed-upwards events $E_1 = A \cup C_1$ and $E_2 = A \cup C_2$:

$$\mu(A)\big(\mu(A)+\mu(C_1)+\mu(C_2)+\mu(B)\big) \ge \big(\mu(A)+\mu(C_1)\big)\big(\mu(A)+\mu(C_2)\big).$$}
\end{remark}

\begin{proof}[Proof of Theorem \ref{Harris+}]
Fix $k$. We proceed by induction on $n$. The case $n=0$ is trivial. Let $\mathbf{v} \in H_{n-1}$, we can identify it with the corresponding vector in $H_n$, which has $n$-th coordinate equal to $0$. Denote by $\mathbf{v}\!\uparrow$ the corresponding vector in $H_n$, which has $n$-th coordinate equal to $1$. Define
\begin{align*}A_0 &:= \{\mathbf{v} \in H_{n-1} : \mathbf{v} \in A,\mathbf{v}\!\uparrow~  \in A\},\\
B_0 &:= \{\mathbf{v} \in H_{n-1} : \mathbf{v} \in B, \mathbf{v}\!\uparrow~ \in B\},\\
C_i^+ &:= \{\mathbf{v} \in H_{n-1} : \mathbf{v} \in C_i, \mathbf{v}\!\uparrow~ \in A\},\\
C_i^\circ &:= \{\mathbf{v} \in H_{n-1} : \mathbf{v} \in C_i, \mathbf{v}\!\uparrow~ \in C_i\},\\
C_i^- &:= \{\mathbf{v} \in H_{n-1} : \mathbf{v} \in B, \mathbf{v}\!\uparrow~ \in C_i\},\\
D &:= \{\mathbf{v} \in H_{n-1} : \mathbf{v} \in B, \mathbf{v}\!\uparrow~ \in A\}.\end{align*}

Using the assumptions in the theorem, we have:
$$H_{n-1} = A_0 \sqcup B_0 \sqcup_{i=1}^k C_i^+ \sqcup_{i=1}^k C_i^\circ \sqcup_{i=1}^k C_i^- \sqcup D \text{ (see Fig. \ref{drawing}).}$$
\begin{center}
\begin{figure}[h]
\centering\includegraphics[width=0.4\linewidth]{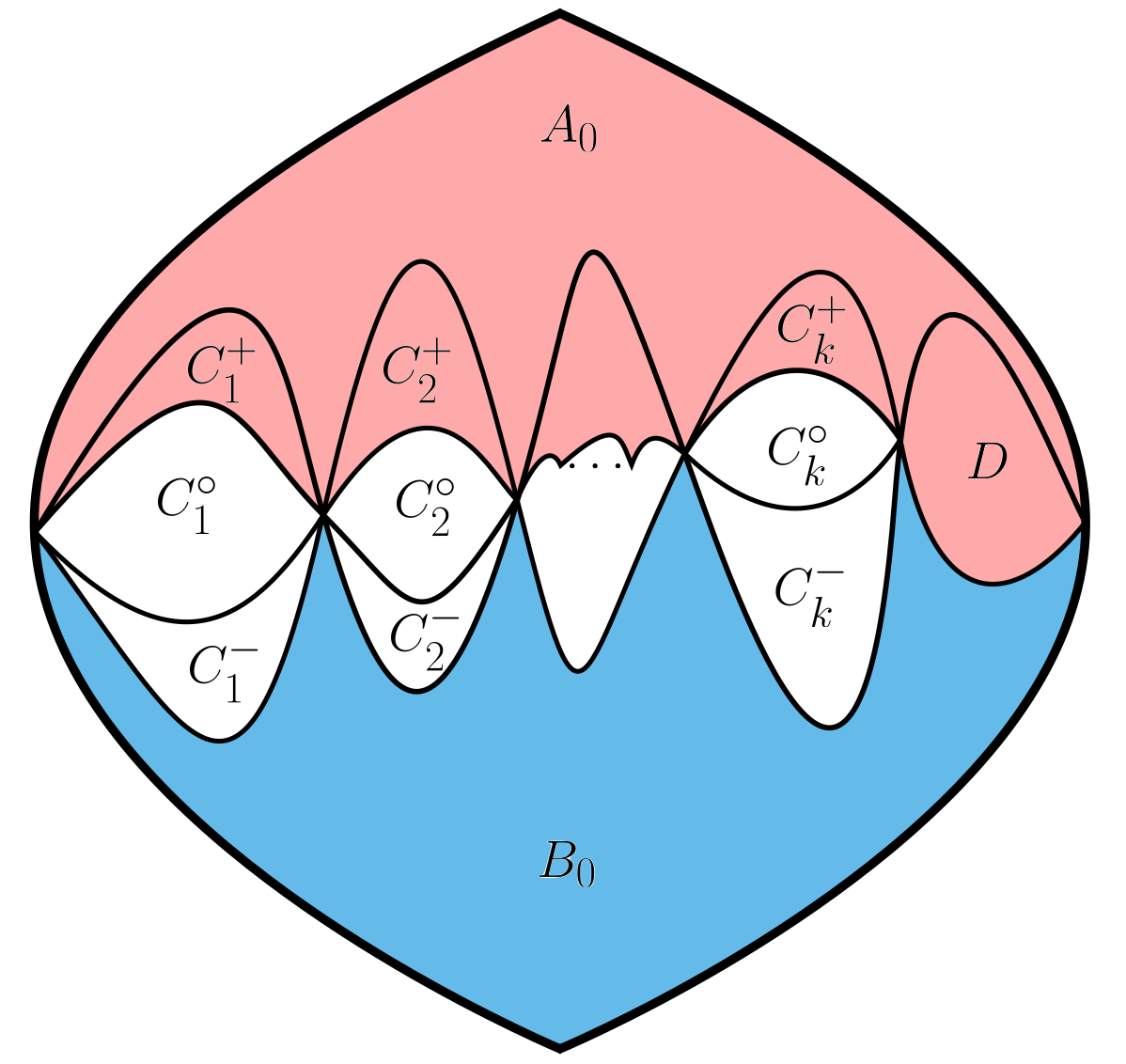}\includegraphics[width=0.4\linewidth]{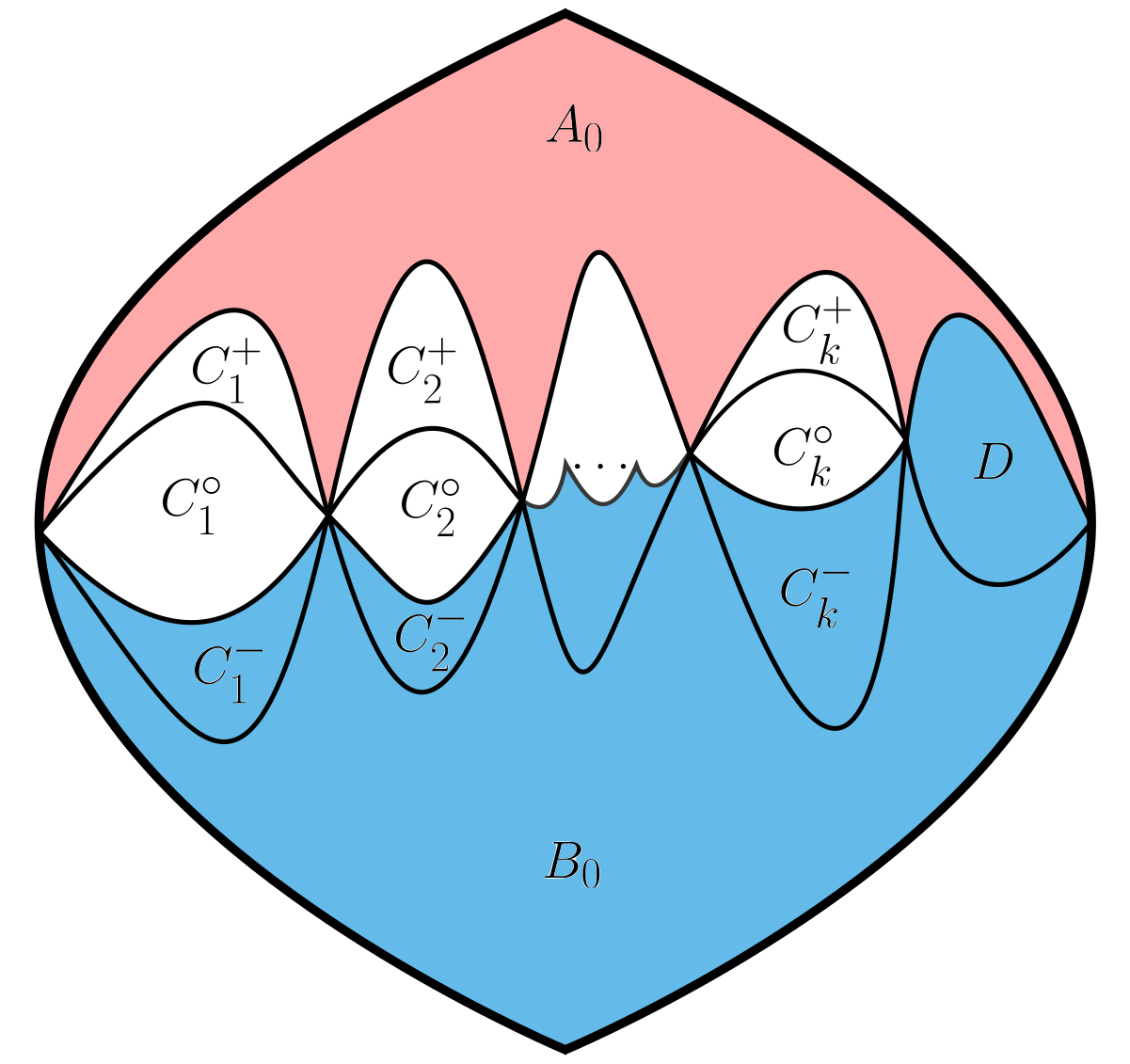}
\caption{Subdivision of $\{0,1\}^{n-1}$}\label{drawing}
\end{figure}
\end{center}

Note that the projection of product measure $\mu$ to $H_{n-1}$ along the $n$-th coordinate is also a product measure. Denote it by $\mu'$. Also, denote 

$$a_0 := \mu'(A_0),\hskip0.8em b_0:=\mu'(b_0),\hskip0.8em c_i^+ := \mu'(C_i^+),\hskip0.8em c_i^\circ := \mu'(C_i^\circ),\hskip0.8em c_i^- := \mu'(C_i^-),\hskip0.8em d := \mu'(D).$$

By the induction hypothesis we have (see first subdivision in Fig. \ref{drawing}):

$$\left(a_0+d+\sum_{i=1}^k c_i^+\right) b_0 ~\ge~ e_2(c_1^\circ+c_1^-, \dots, c_k^\circ+ c_k^-)$$

and (see second subdivision in Fig. \ref{drawing})

$$a_0\left(b_0+d+\sum_{i=1}^k c_i^-\right) ~\ge~ e_2(c_1^\circ+c_1^+, \dots, c_k^\circ+ c_k^+).$$
Let $p:=\mu(H_{n-1})$. We need to show that 
\vskip-0.4cm

\begin{align*}\left(a_0+p\left(d+\sum_{i=1}^k c_i^+\right)\right)&\left(b_0+\left(1-p\right)\left(d+\sum_{i=1}^k c_i^-\right)\right) \\ \ge~ e_2&\left(c_1^\circ+pc_1^-+\left(1-p\right)c_1^+,~ \dots,~ c_k^\circ+ pc_k^-+\left(1-p\right)c_k^+\right).\end{align*}

Note that this inequality is quadratic in $p$ and holds for $p=0$ and $p=1$. Thus it suffices to prove that the coefficient in $p^2$ in the LHS is less than that of the RHS:

$$-\left(d+\sum_{i=1}^k c_i^+\right)\left(d+\sum_{i=1}^k c_i^-\right) \le e_2\left(c_1^--c_1^+, \dots, c_k^--c_k^+\right).$$
And after canceling the terms this can be rewritten as
\vskip-0.4cm

$$-d\left(2d+\sum_{i=1}^k c_i^+ +\sum_{i=1}^k c_i^-\right) - \sum_{i=1}^k c_i^-c_i^+ ~\le~ e_2\left(c_1^-, \dots, c_k^-\right)+e_2\left(c_1^+, \dots, c_k^+\right).$$
This follows since the LHS is nonpositive and the RHS is nonnegative. This completes the induction step, and implies inequality \eqref{E2} for all $n$.
\end{proof}

%\begin{remark}
%The equality case in inequality \eqref{E2} is attained only when two of $C_i$'s are nonempty and $A \cup C_{i_1}$ is independent from $A \cup C_{i_2}$. This suggests further improvements of the inequality \eqref{E2} in the theorem.
%\end{remark}

\section{UI and FUI measures}

Recall the definition of UI and FUI measures introduced in \cite{kahn2022note}. Suppose $X_1, \dots, X_n$ are (dependent) Bernoulli random variables and $\mu$ is their joint distribution. Measure $\mu$ on $H_n$ is called $FUI$ (which stands for \textit{finitely many underlying independents}), if there is a realization of $X_i$'s as increasing functions of independent Bernoulli random variables $Y_1, \dots, Y_m$ for some $m$. Measure $\mu$ is called $UI$, if it is a limit of $FUI$ measures on the same hypercube.

Notice that the FUI and UI are weaker than the FKG property:

\begin{proposition}\textup{\cite[Footnote~1]{kahn2022note}}\label{FKGtoFUI}
All measures $\mu$ with the FKG property are FUI (and therefore UI).
\end{proposition}
\begin{proof}

Let $Z_1, \dots, Z_n$ be i.i.d.\! $U(0, 1)$ random variables. For all $i \in \{1, \dots, n\}$ we recursively define $X_i$ as functions of $Z$'s as follows. Assume that at the $i$-th step, we have 
$$X_j := f_j(Z_1, \dots, Z_{j}) \text{ for all } j<i.$$
Define
\begin{equation}X_i := \begin{cases}0, \text{ if }Z_i < \mu(\mathbf{v}_i = 0 ~|~ \mathbf{v}_j=X_j \text{ for all } 1 \le j < i);\\
1, \text{ otherwise.}
\end{cases}\label{xinz}
\end{equation}
It is easy to see that $\mu$ is the law of $(X_1, \dots, X_n)$. Moreover, the FKG property implies that $X_i$'s are non-decreasing in the $Z_j$'s. 

To prove that $\mu$ is FUI we need to represent $X_i$'s as functions of independent Bernoulli variables. Notice that \eqref{xinz} depends monotonically on a finite (though, exponential in $n$) number of events of form 
$$A(i, \mathbf{v}_1, \dots, \mathbf{v}_{i-1}) := \big\{Z_i < \mu(\mathbf{v}_i = 0 ~|~ \mathbf{v}_j=X_j \text{ for all } 1 \le j < i)\big\}.$$

It is possible to realize indicators of $A(\cdot)$'s as non-decreasing functions of independent, but possibly differently distributed, Bernoulli variables $Y(i, \mathbf{v}_1, \dots, \mathbf{v}_{i-1})$.

%	For $i$ in $\{1, 2, \dots, n\}$ and $w_1, w_2, \dots, w_{i-1} \in \{0, 1\}$ let $Y(i, w_1, w_2, \dots, w_{i-1})$ be a set of independent Bernoulli variables, such that
%	\begin{equation}\mathbf{P}\Bigl(Y(i, w_1, w_2, \dots, w_{i-1}) \Bigr)= \mu(\mathbf{v}_i = 0 ~|~ \mathbf{v}_j=w_j \text{ for all } 1 \le j < i).
%	\end{equation}
%	
%	Denote the set of $Y$'s by $\mathcal{Y}$. We recursively define $X_i$ as functions of $\mathcal{Y}$. On $i$-th step, we already have $X_j := f_j(\mathcal{Y})$ for all $j<i$. We put 
%	
%	$$X_i = Y(i, X_1, \dots, X_{i-1}).$$
%	
%	We are left to prove that $X_i$ is non-decreasing in $\mathcal{Y}$. Indeed, for $i < j$
%	
%	$$\mathbf{P}\bigl(X_i=1~|~Y(j, w_1, w_2, \dots, w_{j-1}) = 1\bigr) = \mu(\mathbf{v}_i=1)$$
%	
%	and for $i \ge j$
%	
%	\begin{multline}\mathbf{P}\bigl(X_i=1~|~Y(j, w_1, w_2, \dots, w_{j-1}) = 1\bigr) = \\
%	\mu(\mathbf{v}_i=1 | X_k=w_k \text{ for all }k<j)\mu(X_k=w_k \text{ for all }k<j) +\\ \mu(\mathbf{v}_i=1 | X_k\ne w_k \text{ for some }k<j)\mu(X_k \ne w_k \text{ for some }k<j).
%	\end{multline}

\vskip-0.4cm

\end{proof}

Proposition~\ref{FKGtoFUI} allows us to generalize Theorem~\ref{Harris+} to all UI-measures. In particular, it holds for all measures with the FKG property.

\begin{theorem}\label{FKG+}
Let $\mu$ be a UI measure on $H_n$, and 
$$H_n = A \sqcup C_1 \sqcup C_2 \sqcup \dots \sqcup C_k \sqcup B \text{ for }k\ge 2$$ such that all sets of the form $A \cup C_i$ are closed upwards. Then

\begin{equation}\mu(A)\mu(B) \ge e_2\big(\mu(C_1), \dots, \mu(C_k)\big).\label{alsoE2}\end{equation}
\end{theorem}
\begin{proof}
Suppose $\mu$ is an FUI measure on $(X_1, \dots, X_n)$. Then we can assume $X_i$'s are binary non-decreasing functions of independent $m$ Bernoulli variables $Y_1, \dots, Y_m$ as in the proof of Proposition \ref{FKGtoFUI}. All sets $A \cup C_j$ are closed upwards in the hypercube generated by $Y_i$'s, so by Theorem~\ref{Harris+} we have inequality \eqref{alsoE2}. For UI measures, inequality \eqref{alsoE2} is obtained as a limit of inequalities for FUI measures.
\end{proof}
\begin{remark}\rm{
Following the original proof in \cite{FKG71}, one may extend Theorem~\ref{FKG+} to general distributive lattices. Note that the proof in \cite{AB, AD} does not extend here. It would be interesting to obtain a functional analog of the equation \eqref{alsoE2}, similar to how the AD inequality serves as a functional analog of the FKG inequality.}
\end{remark}

The following is the main result of Kahn \cite{kahn2022note} and a basis for our main application:

\begin{theorem}\rm{\cite[Corollary 4]{kahn2022note}}\label{KahnFUI}
There are measures on $H_n$ with positive associations which are not FUI.
\end{theorem}

We identify subsets of $\{1, 2, \dots, n\}$ with points in $H_n$. Consider the law $\mu_n$ of the set of fixed points of a uniform permutation $\sigma \in S_n$. It was shown in \cite{FDS} that $\mu_n$ has positive associations. Kahn uses the measure $\mu_3$ to prove Theorem~\ref{KahnFUI}. He writes: ``it seems surprisingly hard to say anything about the law of a UI $\mu$ that uses more than positive association''.

It turns out that Theorem~\ref{FKG+} helps us to extend Kahn's theorem to UI measures. In fact, we use the same measure $\mu_3$. This answers a question dating back to at least 2002 \cite[Question 1]{kahn2022note}. 

\begin{theorem}\label{ans}
There are measures on $H_n$ with positive associations which are not UI.
\end{theorem}
\begin{proof}
Note that $$\mu_3(\{1\})=\mu_3(\{2\})=\mu_3(\{3\})= \mu_3(\{1,2,3\})=\frac{1}{6}$$ and $\mu_3(\varnothing)=\frac{1}{3}$. Consider $A = \{|S| \ge 2\}$, $B = \{S=\varnothing\}$, $C_i = \{S=\{i\}\}$ for $1\le i\le 3$. Suppose $\mu_3$ is $UI$. Then by Theorem~\ref{FKG+} we have 

$$\frac{1}{18} \ge e_2\left(\frac{1}{6}, \frac{1}{6}, \frac{1}{6}\right) = \frac{1}{12},$$
a contradiction. Thus $\mu_3$ is not UI, as desired.
\end{proof}

\section{Applications}
\subsection{Counting graphs}

We give here an application in the style of \cite[Problem 6.5.3]{AS}.

\begin{corollary}
Let $G$ be a uniform random graph on $2n$ labeled vertices and denote by $S$ its set of vertices with degree $\ge n$. Then for every $k$
$$\frac{{2n \choose k}-1}{2{2n \choose k}}\mathbf{P}(|S|=k) \le \mathbf{P}(|S|>k)\mathbf{P}(|S|<k)$$
\end{corollary}
\begin{proof}
Random graphs on $2n$ vertices form a hypercube $H = \{0, 1\}^d$ by inclusion, where $d = {2n \choose 2}$. We can consider events $A$ and $B$ in this hypercube equal to $\{|S|>k\}$ and $\{|S|<k\}$ and events $C_T = \{S=T\}$ indexed by all $k$-subsets $T$ of $\{1, 2, \dots, 2n\}$. All $C_T$ share a probability equal to $\frac{\mathbf{P}(|S|=k)}{{2n\choose k}}$, so applying Theorem~\ref{Harris+}, we get 

$${{2n \choose k} \choose 2}\frac{\mathbf{P}(|S|=k)^2}{{2n \choose k}^2} \le \mathbf{P}(|S|>k)\mathbf{P}(|S|<k).$$
\end{proof}

Note that by using just the Harris--Kleitman inequality, the best we can achieve is 

$$\left\lfloor\frac{{2n \choose k}}{2}\right\rfloor \left\lceil\frac{{2n \choose k}}{2}\right\rceil \frac{\mathbf{P}(|S|=k)^2}{{2n \choose k}^2} \le \mathbf{P}(|S|>k)\mathbf{P}(|S|<k),$$
which is worse by a factor approaching $2$ as $n \to \infty$. In particular, for $k=n$, we have 

$$\mathbf{P}(|S|=n) \le \mathbf{P}(|S|>n)\frac{{2n \choose n}}{\sqrt{{2n \choose n} \choose 2}}.$$ 

This implies 

% \mathbf{P}(|S|=n) * 2 / \frac{{2n \choose n}}{\sqrt{{2n \choose n} \choose 2}}  \le  (1-\mathbf{P}(|S|=n))

$$ \mathbf{P}(|S|=n)  \le \frac{  {2n \choose n} }{ {2n \choose n}+ 2{\sqrt{{2n \choose n} \choose 2} }} \to \sqrt{2} - 1\text{\quad as }n \to \infty.$$ This is an improvement over $\frac{1}{2}$ which follows from the Harris--Kleitman inequality.\footnote{In reality, this number goes to zero, see this \href{https://bit.ly/3V7eC1P}{Mathoverflow answer}. So the inequality is of interest for relatively small $n$.}

\subsection{Percolation}
Theorem~\ref{Harris+} allows us to say more about connectedness events in percolation than the Harris--Kleitman inequality. Consider a graph $G = (V, E)$, where $V=\{1, 2, \dots, n\}$. Consider the percolation on $G$, where each edge $e \in E$ has probability $p_e \in (0,1)$ of surviving, independent of other edges. This gives a spanning subgraph $H \subseteq G$ with probability 
$$\prod_{e \in H} p_e \prod_{e \not\in H}(1-p_e).$$ 

Consider three vertices $1, 2, 3 \in V$. Denote by $\mathbf{P}(123)$ the probability that vertices $1$, $2$ and $3$ lie in the same connected component of $H$. Denote by $\mathbf{P}(12|3)$ the probability that $1$ and $2$ lie in the same connected component, different from the component of $3$. Define $\mathbf{P}(13|2)$ and $\mathbf{P}(1|23)$ analogously. Finally, denote by $\mathbf{P}(1|2|3)$ the probability that all three vertices lie in different connected components.
\begin{corollary}
In the notation above, we have:
\begin{equation}\mathbf{P}(123)\mathbf{P}(1|2|3) \ge \mathbf{P}(12|3)\mathbf{P}(13|2)+\mathbf{P}(12|3)\mathbf{P}(1|23)+\mathbf{P}(13|2)\mathbf{P}(1|23). \label{Percolation}\end{equation}
\end{corollary}
\begin{proof}
Note that events $A = (123)$, $B = (1|2|3)$, $C_1=(1|23)$, $C_2=(13|2)$, $C_3=(12|3)$ satisfy the conditions of Theorem~\ref{Harris+}. The inequality \eqref{Percolation} follows.
\end{proof}

\section*{Acknowledgements}

The author wants to thank his advisor Igor Pak for suggesting the problem and Aleksandr Zimin for fruitful discussions. We also thank Tom Hutchcroft and Jeff Kahn for helpful comments.

\end{document}